\documentclass[11pt, reqno]{amsart}

\usepackage {amsmath, amssymb, amscd, mathrsfs, url,verbatim,enumerate}

\usepackage{multirow,latexsym,graphicx,epstopdf}
\usepackage[dvipsnames,usenames]{color}
\usepackage[colorlinks=true, urlcolor=NavyBlue, linkcolor=NavyBlue, citecolor=NavyBlue]{hyperref}

\usepackage[text={6.5in,9in},centering,letterpaper,dvips]{geometry}
\usepackage[all]{xy}
\usepackage{epsfig}
\usepackage[latin1]{inputenc}
\usepackage{amsmath}
\usepackage{amsfonts}
\usepackage{amssymb}
\usepackage{amsthm}
\usepackage{amscd}
\usepackage{verbatim}
\usepackage{pinlabel}
\usepackage{enumerate}

\usepackage{verbatim}
	
\newtheorem{theorem}{Theorem}[section]
\newtheorem{Theorem}{Theorem}
\newtheorem{lemma}[theorem]{Lemma}
\newtheorem{proposition}[theorem]{Proposition}

\newtheorem{Conjecture}[Theorem]{Conjecture}

\theoremstyle{definition}

\theoremstyle{remark}

\newtheorem{Question}[theorem]{Question}

\newcommand{\HFred}{HF^{\textup{red}}}

\newcommand{\Spinc}{$\mathrm{Spin}^c\;$}

\subjclass[2013]{}

\author[\c{C}a\u{g}r{\i} Karakurt]{\c{C}a\u{g}r{\i} Karakurt}
\address{Department of Mathematics, Bo{\u{g}}azi{\c{c}}i University, Bebek, {\.{I}}stanbul, TR-34342, Turkey}
\email{cagri.karakurt@boun.edu.tr}

\author[Takahiro Oba]{Takahiro Oba}
\address{Department of Mathematics, Tokyo Institute of Technology, 2-12-1 Ookayama, Meguroku, Tokyo 152-8551, Japan}
\email{ oba.t.ac@m.titech.ac.jp}

\author[Takuya Ukida]{Takuya Ukida}
\address{Department of Mathematics, Tokyo Institute of Technology, 2-12-1 Ookayama, Meguroku, Tokyo 152-8551, Japan}
\email{ukida.t.aa@m.titech.ac.jp}

\numberwithin{equation}{section}

\title{Planar Lefschetz fibrations and Stein structures with distinct Ozsv\'ath-Szab\'o invariants on corks}
\date{\today}

\begin{document}
\maketitle

\begin{abstract}

Thanks to a result of Lisca and Mati\'c and a refinement by Plamenevskaya, it is known  that on a 4-manifold with boundary Stein structures with non-isomorphic \Spinc structures induce contact structures with distinct Ozsv\'ath-Szab\'o invariants. Here we give an infinite family of examples showing that converse of Lisca-Mati\'c-Plamenevskaya theorem does not hold in general.  Our examples arise from Mazur type corks. 
\end{abstract}

\section{Introduction}

For any contact structure $\xi$ on a 3-manifold $Y$, let $c^+(\xi) \in HF^+(-Y)$ denote its Ozsv\'ath-Szab\'o invariant. Recall Lisca-Mati\'c-Plamenevskaya theorem:
\begin{theorem}\cite[Theorem 1.2]{LM}  \cite[Theorem 2]{P}
Let $W$ be a smooth compact $4$-manifold with boundary $W$ equipped with two Stein structures $J_1$ and $J_2$ with associated \Spinc structures  $\mathfrak{s}_1$ and $\mathfrak{s}_2$ on $W$, and the induced contact structures $\xi_1$ and $\xi_2$ on $\partial W$. If  $\mathfrak{s}_1$ and $\mathfrak{s}_2$ are not isomorphic then $\xi_1$ and $\xi_2$ are not isotopic; In fact $c^+(\xi_1)\neq c^+(\xi_2)$
\end{theorem}
In the light of the above theorem a natural question to ask is whether the \Spinc structure of a Stein filling completely determines the Ozsv\'ath-Szab\'o invariant of the induced contact structure. An evidence towards a positive answer was provided in a work of  Karakurt \cite[Proposition 1.2]{K} where it was shown that the Ozsv\'ath-Szab\'o invariant depends only on the first Chern class of the Stein filling on $W$ when the total space of the filling is a special type of plumbing. Our main result suggests that the answer is in general negative. To state it  let $\pi: HF^+(-\partial Y)\to  \HFred (-\partial Y)$ be the natural projection map from the plus flavor to reduced Heegaard Floer homology. 

\begin{theorem}\label{theo:main} There exists an infinite family $\{W^n\,: \, n\in \mathbb{N} \}$ of compact contractible $4$-manifolds with boundary and Stein structures $J_1^n$ and $J_2^n$ on $W^n$ satisfying the following properties:
\begin{enumerate}
\item The \Spinc structures $\mathfrak{s}_1^n$ and $\mathfrak{s}_2^n$  associated to $J_{1}^{n}$ and $J_{2}^{n}$, respectively, are the same for every $n\in \mathbb{N}$.
\item  The induced contact structures $\xi_1^n$ and $\xi_2^n$ on $\partial W^{n}$ have distinct Ozsv\'ath-Szab\'o invariants, in fact $\pi (c^+(\xi^n_1))\neq 0$ and $\pi (c^+(\xi^n_2))= 0$, for every $n\in \mathbb{N}$.
\item the Casson invariant of $\partial W^n$ is given by $\lambda (\partial W^n)=2n$ for  every $n\in \mathbb{N}$.
\item $\partial W^n$ is irreducible for every $n\in \mathbb{N}$.
\end{enumerate}
\end{theorem}

Our examples $W^n$ are Mazur type manifolds obtained from the symmetric link $L^n$ in Figure \ref{fig:corkWnsymmetric} by putting a dot  on one of the components and attaching a $0$-framed $2$-handle  to the other component as in Figure \ref{fig:corkWn}. Note that  the manifold $W^1$ is the Akbulut cork. A Stein structure $J_1^n$ on $W^n$ can immediately  be obtained by drawing a Legendrian representative of the attaching circle of the $2$-handle and stabilizing as necessary to make the framing one less than the Thurston-Bennequin framing. Even though the choice of stabilizations is not unique, and different stabilizations potentially yield  Stein structures with distinct Ozsv\'ath-Szab\'o invariants,  the direct computation of these invariants does not seem plausible. Hence we take a different approach and construct the second Stein structure $J_2^n$ using the Loi-Piergallini-Akbulut-Ozbagci correspondence between Stein structures and positive allowable Lefschetz fibrations (PALFs in short). Our  key observation is that $W^n$  admits a planar PALF, that is, a PALF with planar fiber. This was already shown for the Akbulut cork $W^1$ by Ukida \cite{U}. The main result for $W^1$   then immediately follows by bringing together some known facts in the literature (see our proof below). One can easily promote  Ukida's example to an infinite family by repeatedly taking boundary sums of $W^1$. The irreducibility of $\partial W^n$  shows that our examples do not arise in this manner. In the body of our work we generalize Ukida's planar PALF construction to $W^n$, and compute the Casson invariants to distinguish $\partial W^n$'s.  Along the way we also prove that  $\partial W^n$ is obtained from $S^3$ by $1/n$-surgery on a knot, a fact we find interesting in its own right. 

\begin{figure}[h]
	\includegraphics[width=0.50 \textwidth]{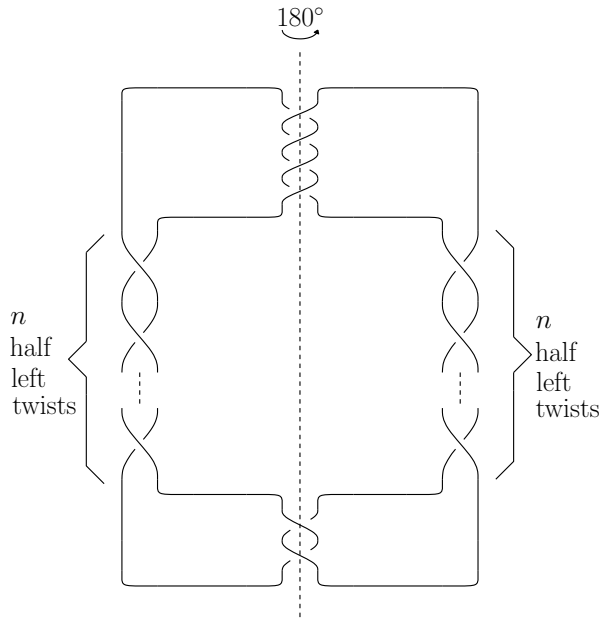}
	\caption{Symmetric picture of $L^n$. The indicated involution exchanges the components.}
	\label{fig:corkWnsymmetric}
\end{figure}

\begin{figure}[h]
	\includegraphics[width=0.40\textwidth]{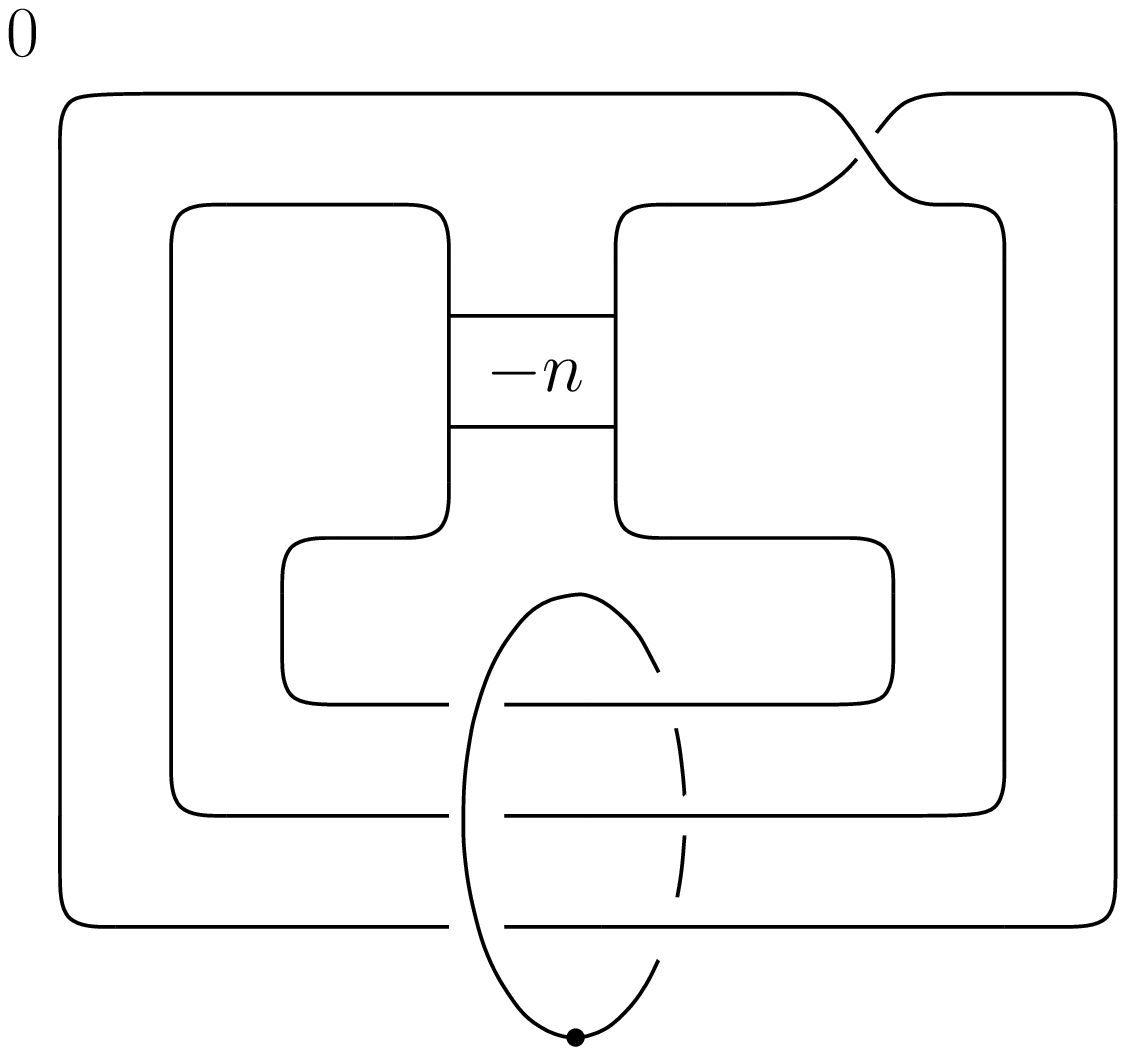}
	\caption{The handlebody $W^n=W(L^n)$. The box indicates $n$ full left twists.}
	\label{fig:corkWn}
\end{figure}

\section{Proof of the main theorem}
First recall the terminology from \cite{AK}.
Let $L$ be a link in $S^3$ with two components $K_1\cup K_2$. We say that  $L$  is an \emph{admissable} link if it satisfies the following conditions:
\begin{enumerate}
\item  Both  $K_1$ and $K_2$ are unknotted.
\item An involution of $S^3$  exchanges $K_1$ and $K_2$.
\item The linking number of $K_1$ and $K_2$ is $\pm 1$.
\item  Carve out a disk bounded by $K_1$ and regard  $K_2 \subset S^1\times S^2=\partial (S^1\times B^3)$ equipped with the unique Stein fillable contact structure. Then the  maximal Thurston-Bennequin number of $K_2$ is at least $+1$.
\end{enumerate}
From an admissible link, we can construct an obvious contractible Stein handlebody $W(L)$ by putting a dot on $K_1$, and attaching a $2$-handle along  some Legendrian representative $K_2$ with framing one less than the Thurston-Bennequin framing (this is possible thanks to the last condition). 

As in the introduction, let $L^n$ be the link  given in Figure \ref{fig:corkWnsymmetric}, and let $W^n:=W(L^n)$ denote the corresponding handlebody obtained by putting a dot on one of the components and $0$ on the other one as in Figure \ref{fig:corkWn}.

\begin{proposition}\label{prop:admis} For every $n\in\mathbb{N}$, the link $L^n$  is admissible. 
\end{proposition}
\begin{proof}
In Figure \ref{fig:corkWnsymmetric} the $180^\circ$ rotation about the dashed axis exchanges the components of $L^n$.  It is also clear from the figure that  both components of $L^n$ are unknotted and  the linking number of these components is $\pm 1$.  We must check that the handlebody  $W^n$ is Stein. By Eliashberg's characterization, it suffices to show that the attaching circle of the $2$-handle has maximal Thurston Bennequin number $TB\geq 1$ in $S^1\times S^2$. In Figure \ref{fig:corkWnStein}, we draw a Legendrian representative of the attaching circle of the $2$-handle on $S^1\times S^2$. From the figure we see that the writhe is $2n+1$ and half the number of cusps is $2n-1$, implying that $TB\geq 2$. Hence a stabilization of the figure  gives  a Stein handlebody picture of $W^n$.

\begin{figure}[h]
	\includegraphics[width=0.45\textwidth]{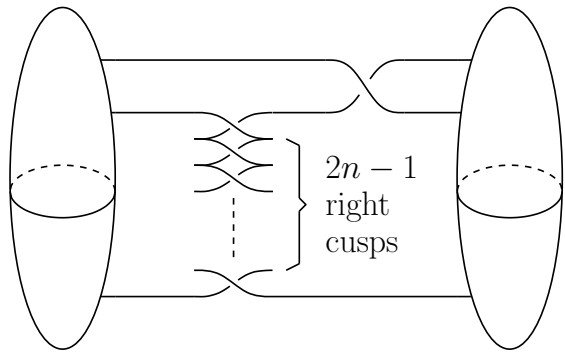}
	\caption{Cork $W^n$ as a Stein handlebody (Need to stabilize the $2$-handle once).}
	\label{fig:corkWnStein}
\end{figure}

\end{proof}
Denote the corresponding Stein structure on $W^n$ (for any choice of stabilization) by $J_1^n$, and the induced contact structure on $\partial W^n$ by $\xi_1^n$. The following result shows that $\pi (c^+(\xi_1^n))\neq 0$.

\begin{theorem}\cite[Theorem 4.1]{AK}\label{t:AK} Let $L$ be an admissible link, $W(L)$ be the corresponding Stein handlebody and $\xi$ the induced contact structure on $\partial W(L)$. Then $\pi (c^+(\xi))\neq 0$.
\end{theorem}

It is important for the above theorem that the Stein structure is the one coming from the handlebody picture associated to an admissible link. 

\begin{proposition}\label{prop:planar}
The manifold $W^n$ admits a planar PALF for every $n\in \mathbb{N}$.
\end{proposition}

\begin{proof}
For $n=1$, this result was proved by Ukida in \cite{U}. We generalize Ukida's argument in an obvious manner.  We apply the handlebody moves indicated in Figure \ref{fig:CorkWn_comp}. Clearly the last diagram gives the total space of PALF whose fibers are disks with $n+3$ holes and monodromy is the following product of right handed Dehn twists $t_a t_b t_c t_{d_1}\cdots t_{d_n}$ where $a$, $b$, $c$, $d_1\dots d_n$ are the curves indicated in Figure \ref{fig:corkWnPALF}. 
\end{proof}

\begin{figure}[h]
	\includegraphics[width=1.0\textwidth]{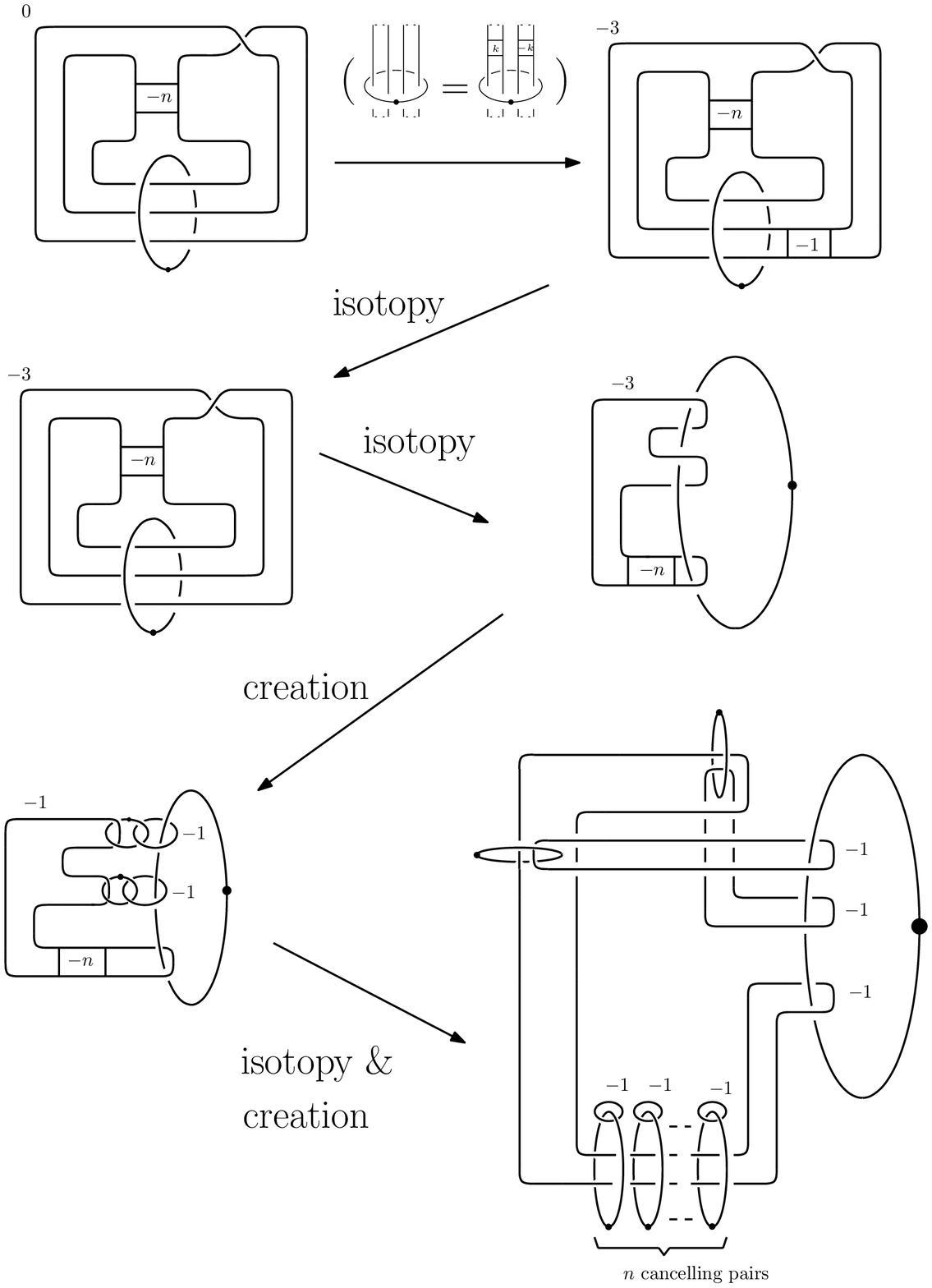}
	\caption{Handlebody moves applied to $W^n$}
	\label{fig:CorkWn_comp}
\end{figure}

\begin{figure}[h]
	\includegraphics[width=0.55\textwidth]{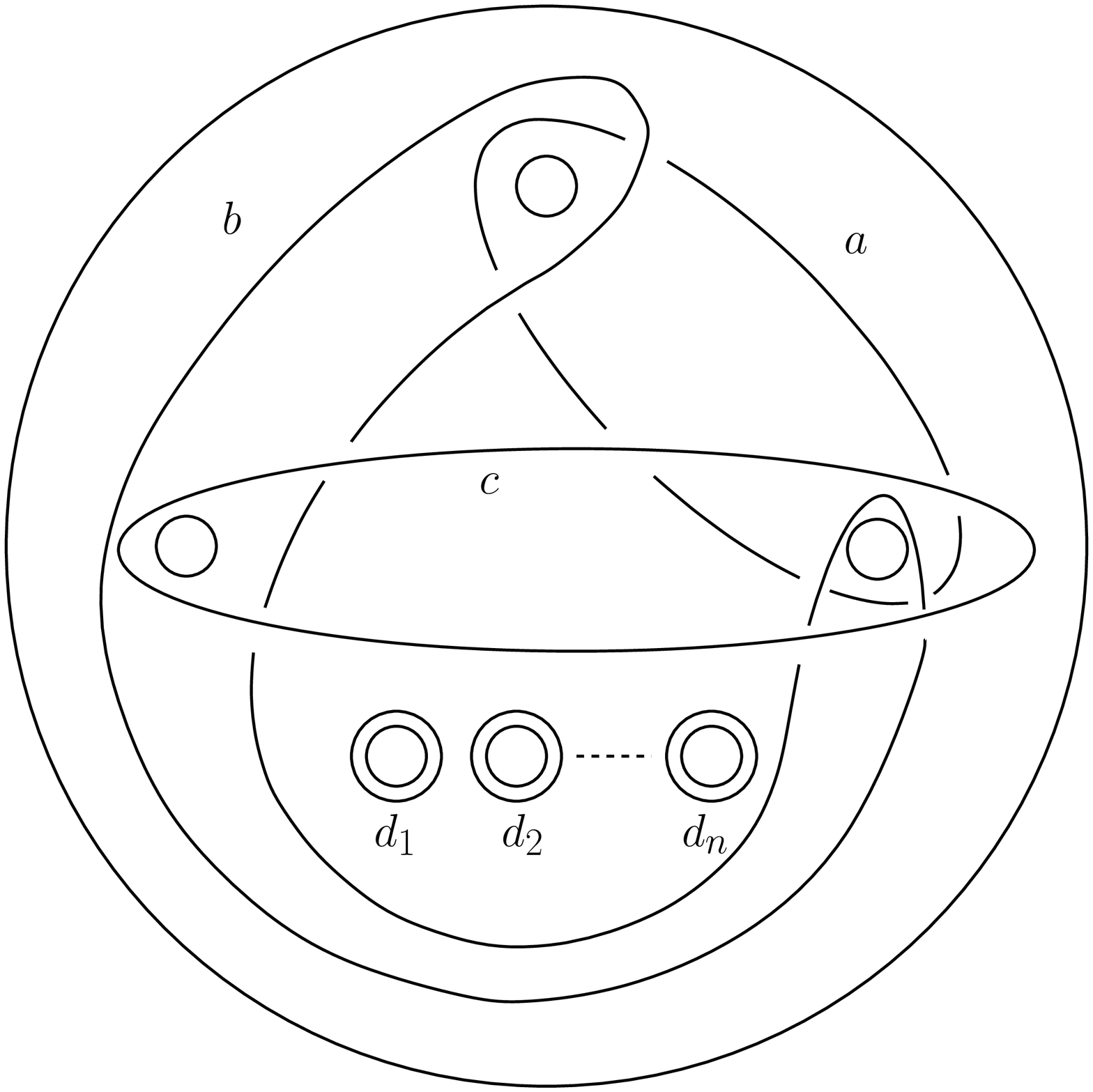}
	\caption{Our planar PALF on $W^n$}
	\label{fig:corkWnPALF}
\end{figure}

Now that we know $W^n$ admits a planar PALF, by results of Loi-Piergallini \cite{LP} and Akbulut-Ozbagci \cite{AO} there is a corresponding Stein structure on $W^n$ which we denote by $J_2^n$. Let $\xi_2^n$ be the induced contact structure on $\partial W^n$.  Note that $\xi_2^n$ is supported by a planar open book.  The next result which is due to Ozsv\'ath-Stipsicz-Szab\'o implies that $\pi ( c^+(\xi_2^n))=0$

\begin{theorem} \cite[Theorem 1.2]{OSS}\label{t:OSS}
Let $Y$ be a $3$-manifold and $\xi$ a contact structure on $Y$. 
Suppose that $\xi$ is supported by a planar open book decomposition. Then $\pi ( c^+(\xi))=0$. 
\end{theorem}

We have just observed that the Ozsv\'ath-Szab\'o invariants of $\xi_1^n$ and $\xi_2^n$ satisfy the required properties. It is clear that the induced \Spinc structures $\mathfrak{s}^n_1$ and $\mathfrak{s}^n_2$ are the same since $W^n$ is contractible. To prove the rest of the theorem first we observe that the boundary of each $W^n$ is the manifold $S^3_{1/n}(K)$ which is obtained from $S^3$ by $1/n$-surgery on the knot $K$ on the left hand side of  Figure \ref{fig:knot}.

\begin{figure}
	\includegraphics[width=.8\textwidth]{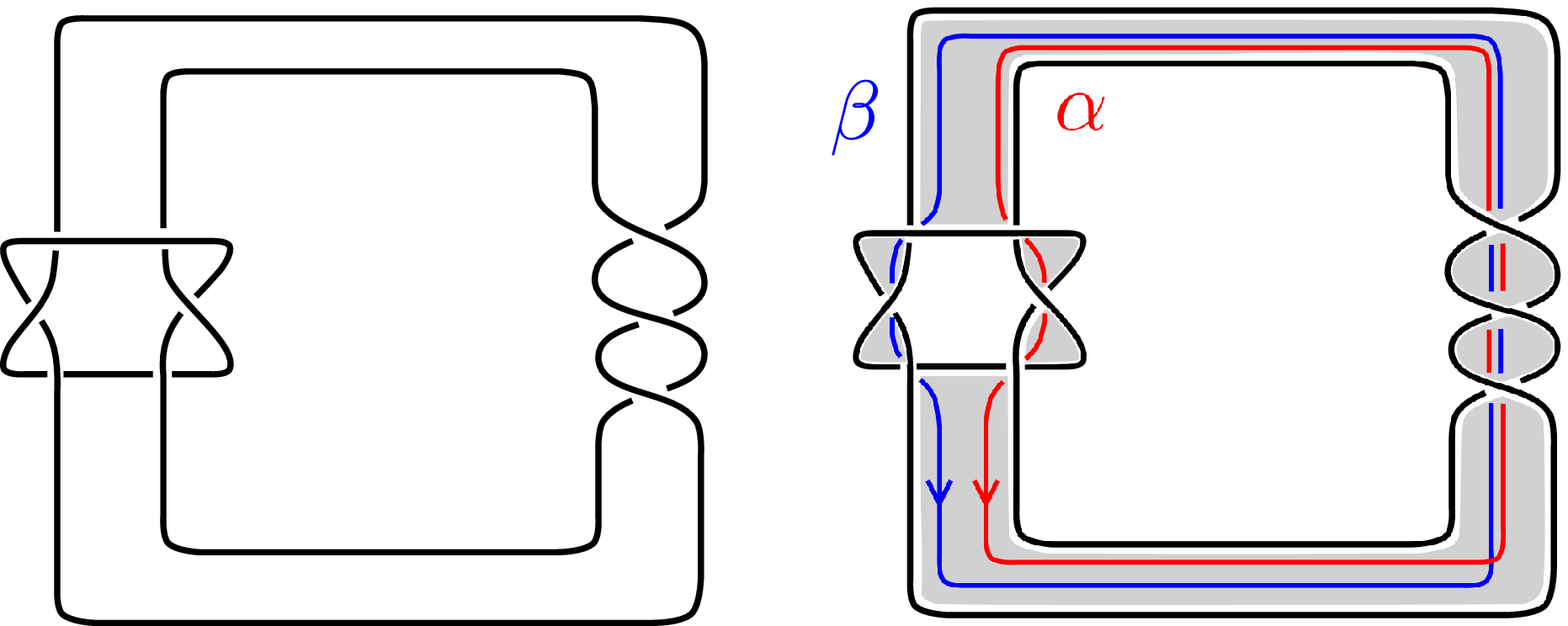}
	\caption{The knot $K$ is on the left. The Seifert surface of $K$ together with its homology generators are on the right.}
	\label{fig:knot}
\end{figure}

\begin{lemma}
We have $\partial W^n=S^3_{1/n}(K)$ for all $n\in \mathbb{N}$.
\end{lemma} 

\begin{proof} 
This was proved for $n=1$ by Akbulut and Kirby \cite[Proposition 1-(3)] {AKi}. One can easily modify their argument to see the proof in the general case. Alternatively we can apply the handlebody moves depicted in Figure \ref{fig:boundary_Wn} and Figure \ref{fig:boundary_Wn_2} to show that $\partial W^n$ is obtained from $S^3$ by $1/n$-surgery on a knot. It is easy to see that the knots in Figure \ref{fig:knot} and at the end of Figure \ref{fig:boundary_Wn_2}  are isotopic.
\end{proof}
\begin{figure}[h]
	\includegraphics[width=.80\textwidth]{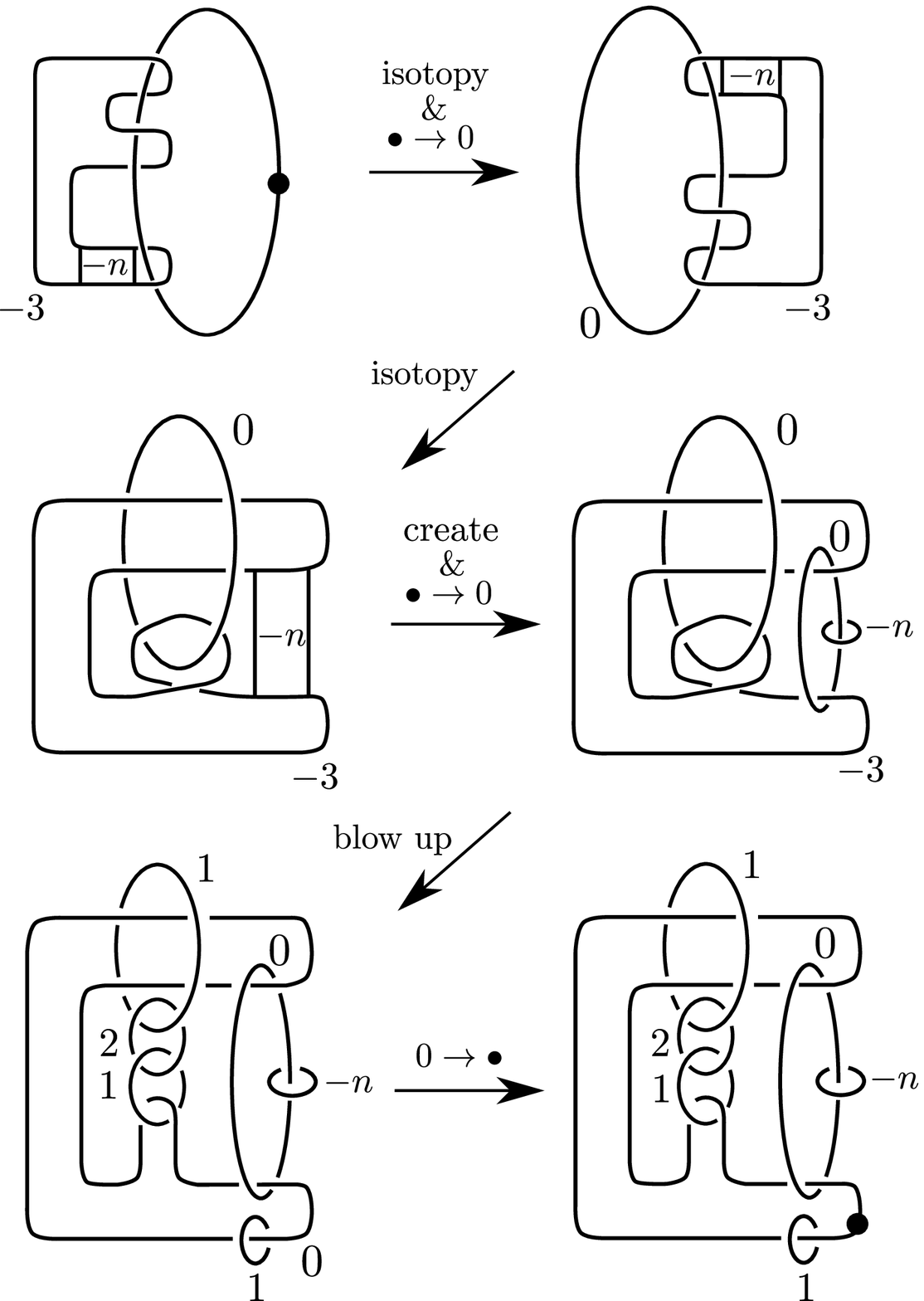}
	\caption{}
	\label{fig:boundary_Wn}
\end{figure}

\begin{figure}[h]
	\includegraphics[width=1.00\textwidth]{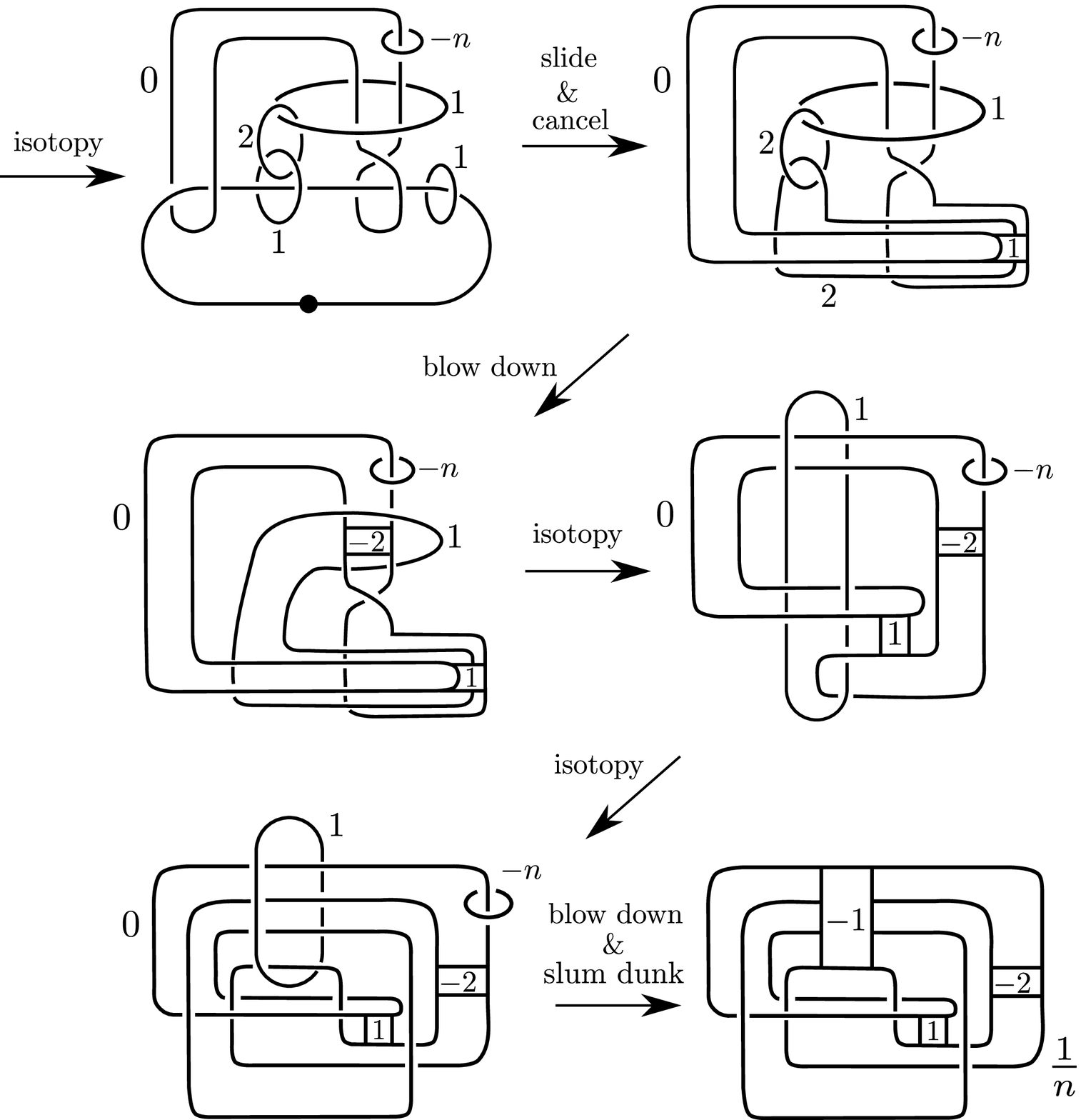}
	\caption{}
	\label{fig:boundary_Wn_2}
\end{figure}

Now the irreducibility of $\partial W^n$ follows from a result of Gordon and Luecke \cite{GL} which says that if a reducible manifold appears as a surgery on a knot in $S^3$ then one of its summands must be a lens space. Since $\partial W^n$ is an integral homology sphere, it cannot have any non-trivial lens space summands. 

\begin{lemma}
The Alexander polynomial of the knot $K$ is given by $\Delta_K(t)=2t^2-5t+2$.
\end{lemma}	

\begin{proof}
We use the Seifert surface of $K$ that is indicated on the right hand side of Figure \ref{fig:knot}. With respect to the homology generators $\alpha,\;\beta$, the Seifert matrix is given by
$$ S=\left [ \begin{tabular}{rr}
3 & -1\\
-2&0
\end{tabular} \right ].$$
Then the Alexander polynomial is $\Delta_K(t)=\mathrm{Det}(S-tS^T)=2t^2-5t+2$.
\end{proof}

From the above Lemma we conclude that the Casson invariant of $\partial W^n$ is given by $$\lambda(\partial W^n)= \frac{n}{2}\Delta_K''(1)=2n,$$ which finishes the proof of Theorem \ref{theo:main}. 

\section{Final remarks}
We anticipate that there are lots of  symmetric links that give rise to contractible manifolds with distinct Stein structures.  For example in \cite{O}, Oba constructed an infinite family of planar PALFs on Mazur type contractible manifolds. We do not know whether the links defining Oba's manifold are admissible. If they are, then they could  be used in the proof instead of $L^n$. 

Finally, our techniques can only  distinguish those Stein structures supported by  planar PALFs from those which cannot, so we are unable to  detect more than $3$ distinct Stein structures on a contractible manifold. Hence the following is still an interesting open problem
\begin{Question}
Given $k\geq 3$, is there a contractible manifold $4$-manifold (with irreducible boundary) which admits $k$ distinct Stein structures?
\end{Question} 

\section{Acknowlegments}
This project started when Karakurt and Oba participated in $23^{\mathrm{rd}}$ G\"okova geometry-topology conference. We would like to thank the organizers for the stimulating atmosphere. 
Also, we would like to thank Takahiro Kitayama for helpful comments about $3$-manifolds. 
During the course of the project,  \c{C}a\u{g}r{\i} Karakurt was supported by a TUBITAK grant BIDEB 2232 No: 115C005, and 
Takahiro Oba was supported by JSPS KAKENHI Grant Number 15J05214. 

\bibliography{References}
\bibliographystyle{amsalpha}

\end{document}